\documentclass[oneside,10pt,leqno]{amsart}
\usepackage{amssymb,amsmath,latexsym,amscd}
\usepackage{array,hhline}

\usepackage{longtable}

\def\gg{\mathfrak{g}}
\def\gh{\mathfrak{h}}

\def\gm{\mathfrak{m}}

\def\C{\mathbb{C}}

\def\R{\mathbb{R}}

\def\cD{\mathcal{D}}

\def\Ad{{\rm Ad}}
\def\ad{{\rm ad}\,}

\newtheorem{theorem}[equation]{Theorem}

\newtheorem{lemma}[equation]{Lemma}

\newtheorem{corollary}[equation]{Corollary}

\newtheorem{proposition}[equation]{Proposition}

\def\sideremark#1{\ifvmode\leavevmode\fi\vadjust{\vbox to0pt{\vss
 \hbox to 0pt{\hskip\hsize\hskip1em
\vbox{\hsize2cm\tiny\raggedright\pretolerance10000 
 \noindent #1\hfill}\hss}\vbox to8pt{\vfil}\vss}}} 
\newcommand \<{\langle}
\renewcommand \>{\rangle}
\newcommand \ip{\<\cdot,\cdot\>}

\newcommand \crown{{\bf{cr}}}

\begin{document}

\title{Families of Geodesic Orbit Spaces and Related Pseudo--Riemannian 
Manifolds}

\author{Joseph A. Wolf}\thanks{Research partially supported by a Simons
Foundation grant}
\address{Department of Mathematics \\ University of California, Berkeley \\
	CA 94720--3840, U.S.A.} \email{jawolf@math.berkeley.edu}

\date{file last edited 02 November 2022}

\subjclass[2010]{22E45, 43A80, 32M15, 53B30, 53B35}

\keywords{geodesic orbit space, real form family, weakly symmetric space, 
naturally reductive space, commutative space, D'Atri space, pseudo--riemannian 
manifold, homogeneous manifold}

\begin{abstract}
Two homogeneous pseudo--riemannian manifolds $(G/H, ds^2)$ and
$(G'/H', ds'^2)$ belong to the same {\em real form family} if
their complexifications $(G_\C/H_\C, ds_\C^2)$ and $(G_\C'/H_\C', ds_\C'^2)$
are isometric.  The point is that in many cases a particular space
$(G/H, ds^2)$ has interesting properties, and 
those properties hold for the spaces in its real form family.
Here we prove that if $(G/H, ds^2)$ is a geodesic orbit space
with a reductive decomposition $\gg = \gh + \gm$, then the same holds 
all the members of its real form family.  In particular our understanding
of compact geodesic orbit riemannian manifolds gives information on 
geodesic orbit pseudo--riemannian manifolds.  We also prove 
similar results for naturally reductive spaces, for commutative 
spaces, and in most cases for weakly symmetric spaces.
We end with a discussion of inclusions of these real form families,
a discussion of D'Atri spaces, and a number of open problems.
\end{abstract}

\maketitle

\section{Introduction}\label{sec1}
\setcounter{equation}{0}

Let $(G/H,ds^2)$ be a homogeneous pseudo--riemannian manifold.
For convenience of exposition we assume that $M$ and $G$ are connected.
We have the complexification $(G_\C/H_\C,ds_\C^2)$ where $ds^2$
is extended by complex bilinearity on every tangent space.  
$(G_\C/H_\C,ds_\C^2)$
is pseudo--riemannian of signature $(n,n)$ where $n = \dim M$.
The  {\bf real form family} of $(G/H,ds^2)$ consists of
all pseudo--riemannian manifolds 
$(G'/H',ds'^2)$ with the same (up to isometry) complexification 
$(G_\C/H_\C,ds_\C^2)$.  Following \cite{WC2018} we write
$\{\{(G/H,ds^2)\}\}$ for the real form family of $(G/H,ds^2)$.

There is a little bit of ambiguity in the literature.  In this paper 
$(G_\C/H_\C,ds_\C^2)$ does not belong to $\{\{(G/H,ds^2)\}\}$;
we refer to $(G_\C/H_\C,ds_\C^2)$ as the {\bf crown} of 
$\{\{(G/H,ds^2)\}\}$ and write $(G_\C/H_\C,ds_\C^2) = 
\crown\{\{(G/H,ds^2)\}\}$.

From now on,
suppose that we have a reductive decomposition: $\gg = \gh + \gm$ with
$\Ad(H)\gm = \gm$.  Then $\gm$ represents the tangent space at 
$o = 1H$ and $\ip$ denotes the inner product on $\gm$
defined by $ds^2$.

If $(G'/H',ds'^2) \in \{\{(G/H,ds^2)\}\}$ then
$\{\{(G/H,ds^2)\}\} = \{\{(G'/H',ds'^2)\}\}$.  In particular 
$(G/H,ds^2)$ and $(G'/H',ds'^2)$ have the same (up to isometry)
complexification,
$\crown\{\{(G/H,ds^2)\}\} = \crown\{\{(G'/H',ds'^2)\}\}$.  Further,
we have a reductive decomposition $\gg' = \gh' + \gm'$ stable under 
an involutive isometry $\theta$ of $G'$ such that $(G,H)$ is the 
corresponding real form of $(G',H')$.  This is reversible because 
$\theta$ defines
on involutive isometry of $(G/H,ds^2)$ that preserves $H$ and
such that $(G',H')$ is the corresponding real form of $(G_\C,H_\C)$.
Thus ``belongs to the real form family'' is an equivalence relation.

A nonzero element $\eta \in \gg$ is a 
{\bf geodesic vector} (at $o$)
if $t \mapsto \exp(t\eta)o$ is a geodesic.  A geodesic $t \mapsto \gamma(t)$
is called {\bf homogeneous} if it comes from a geodesic vector as above.
$(G/H,ds^2)$ is a {\bf geodesic orbit space} or {\bf GO space} if 
every geodesic on $(G/H,ds^2)$ is homogeneous.   

\begin{proposition}\label{char-go} {\rm {\bf Geodesic Lemma}}\;
Let $(G/H,ds^2)$ be a homogeneous pseudo--riemannian manifold with a
reductive decomposition $\gg = \gh + \gm$, $[\gh,\gm] \subset \gm$.  Then
$(G/H,ds^2)$ is a geodesic orbit space if and only if
\begin{equation}\label{basic}
\text{given } \xi \in \gm \text{ there exist } \alpha \in \gh \text{ and }
c \in \R \text{ such that } \langle [\xi + \alpha,\zeta]_\gm,\xi \rangle = 
c\,\langle \xi,\zeta \rangle \text{ for all } \zeta \in \gm.
\end{equation}
Further, if $c \ne 0$ then the geodesic is null.
\end{proposition}

For the notion of homogeneous geodesic and the formula (\ref{basic}) 
characterizing geodesic vectors in the riemannian case see \cite{KV1991}.
The pseudo--riemannian case appeared in \cite{FMP2005} and \cite{P2006},
but without a proof. The correct mathematical formulation with the proof was
given in \cite{DK2007}.

We are going to study the structure of real form families of geodesic 
orbit spaces (in \S 3) using the form of the Geodesic Lemma.  Then we look
at corresponding structural matters for naturally reductive spaces (in \S 4),
for commutative spaces (in \S 5), for weakly symmetric spaces (in \S 6),
and for D'Atri spaces (in \$ 7), ending with a list of open problems.

\section{The Moduli Spaces}\label{sec2}
\setcounter{equation}{0}

The moduli spaces $\Omega$ and $\Omega_\C$ will allow us to carry the
$GO$ property between various spaces in a real form family.  Define
real and complex polynomials
\begin{equation}\label{def-poly}
\begin{aligned}
&\varphi: \gm + \gh + \gm \to \R \text{ by }
	\varphi(\xi,\alpha,\zeta) = \< [\xi + \alpha,\zeta]_\gm,\xi \>
		- c\< \xi,\zeta\> \text{ and } \\
&\varphi_\C: \gm_\C + \gh_\C + \gm_\C \to \C \text{ by }
        \varphi_\C(\xi,\alpha,\zeta) = \< [\xi + \alpha,\zeta]_{\gm_\C},\xi \>
                - c\< \xi,\zeta\>
\end{aligned}
\end{equation}
where $\ip$ extends from $\gm + \gm$ to $\gm_\C + \gm_\C$ by complex 
bilinearity. Further, $c = 0$ whenever $\< \xi,\xi\> \ne 0$, in other words
$c \ne 0$ only when $\xi$ is null.  Note that $\varphi_\C(\xi,\alpha,\zeta)$ 
is linear both in $\zeta$ and in $\alpha$, and if $c = 0$ it is quadratic in 
$\xi$.  Define subvarieties
\begin{equation}\label{def-var}
\begin{aligned}
&\Omega = \{(\xi,\alpha) \in (\gm + \gh) \mid
	\varphi(\xi,\alpha,\zeta) = 0 \text{ for every } \zeta \in \gm\}
	\text{ and }\\
&\Omega_\C = \{(\xi,\alpha) \in (\gm_\C + \gh_\C) \mid
	\varphi_\C(\xi,\alpha,\zeta) = 0 \text{ for every } \zeta \in \gm_\C\}.
\end{aligned}
\end{equation}
As one might guess from the notation we have

\begin{proposition}\label{realform}
The real affine variety $\Omega = \Omega_\C \cap (\gm + \gh + \gm)$, and
it is a real form of the complex affine variety $\Omega_\C$.  In other 
words $\Omega_\C$ is the complexification of $\Omega$.
In particular, if $f$ is a holomorphic function on $\Omega_\C$ and
$f|_\Omega \equiv 0$, then $f \equiv 0$.
\end{proposition}

\begin{proof}
As $\varphi$ and $\varphi_\C$ are linear in $\zeta$ we can 
replace the ``every $\zeta$'' conditions in {\rm (\ref{def-var})}
by ``$\{\zeta_1, \dots , \zeta_\ell\}$'' where
$\{\zeta_1, \dots , \zeta_\ell\}$ is a basis of $\gm$.  Now
$\Omega$ is defined by the $\ell$ real polynomial functions
$\varphi_j : (\xi,\alpha) \mapsto \varphi(\xi,\alpha,\zeta_j)$
on $\gm + \gh$,
and $\Omega_\C$  is defined by the $\ell$ complex polynomial function
$\varphi_{j;\C} : (\xi,\alpha) \mapsto \varphi_\C(\xi,\alpha,\zeta_j)$
on $\gm_\C + \gh_\C$\,.  As
$\varphi_j = {\varphi_{j;\C}}|_{(\gm + \gh)}$ the Proposition is
immediate.
\end{proof}

\section{Geodesic Orbit Spaces}\label{sec3}
\setcounter{equation}{0}

We start by pinning down the moduli spaces $\Omega$ and $\Omega_\C$
for the geodesic orbit case.
From the definitions

\begin{lemma}\label{geo-vector}
In the notation of {\rm (\ref{def-var})}, 
$\xi + \alpha$ is a geodesic vector for $(G/H,ds^2)$ if and only
if $\xi + \alpha \in \Omega$, and $\xi + \alpha$ is a 
geodesic vector for the crown $\crown \{\{(G/H,ds^2)\}\}$ 
if and only if $\xi + \alpha \in \Omega_\C$.
\end{lemma}

Thus we have a minor reformulation of the Geodesic Lemma 
(Proposition \ref{char-go}), as follows.

\begin{proposition}\label{geo-space}
In the notation of {\rm (\ref{def-var})},
$(G/H,ds^2)$ is a geodesic orbit space if and only if, for every 
$\xi \in \gm$ there is an $\alpha \in \gh$ with $(\xi,\alpha) \in \Omega$,
and  the crown $\crown\{\{(G/H,ds^2)\}\}$ is a geodesic orbit space if 
and only if, for every $\xi \in \gm_\C$ there is an $\alpha \in \gh_\C$
with $(\xi,\alpha) \in \Omega_\C$.
\end{proposition}

Now we combine Propositions \ref{realform} and \ref{geo-space}:

\begin{theorem}\label{GO-theorem}
In the notation of {\rm (\ref{def-var})},
$(G/H,ds^2)$ is a geodesic orbit space if and only if 
the crown $\crown\{\{(G/H,ds^2)\}\}$
is a geodesic orbit space.  In particular, if $(G'/H',ds'^2) \in
\{\{(G/H,ds^2)\}\}$, then $(G/H,ds^2)$ is a geodesic orbit space if 
and only if $(G'/H',ds'^2)$ is a geodesic orbit space.
\end{theorem}

\begin{proof}
Proposition \ref{geo-space} says that $(G/H,ds^2)$ is a geodesic orbit space 
if and only if the projection $\pi: \Omega \to \gm$, by 
$\pi(\xi,\alpha) = \xi$, is surjective; and also $(G_\C/H_\C,ds_\C^2)$ is a
geodesic orbit space if and only if the projection 
$\pi_\C: \Omega_\C \to \gm_\C$, by $\pi_\C(\xi,\alpha) = \xi$, is surjective.
But Proposition \ref{realform} ensures that $\pi$ is surjective if and only 
if $\pi_\C$ is surjective.  That proves the first assertion.  
Since $\{\{(G'/H',ds'^2)\}\} = \{\{(G/H,ds^2)\}\}$ the corresponding
$\pi' : \Omega' \to \gm'$ is surjective if and only if $\pi_\C$ is 
surjective.  The second assertion follows.
\end{proof}

\section{Naturally Reductive Spaces}\label{sec4}
\setcounter{equation}{0}

A homogeneous space $(G/H,ds^2)$ with a reductive decomposition 
$\gg = \gh + \gm$, $\Ad(H)\gm = \gm$, is called {\bf naturally reductive}
if 
\begin{equation}\label{def-nat}
\text{if } \xi \in \gm \text{ then } t \mapsto \exp(t\xi)H
\text{ is a complete geodesic in } (G/H,ds^2).
\end{equation}
The Lie algebra formulation of (\ref{def-nat}) is
\begin{equation}\label{nat-basic}
\<[\xi,\eta]_\gm,\zeta\> + \<\eta, [\xi,\zeta]_\gm\> = 0 \text{ for all }
\xi, \eta, \zeta \in \gm.
\end{equation}
The case $\zeta = \xi$ is $\<[\xi,\eta]_\gm,\zeta\> + \<\eta, [\xi,\zeta]_\gm\>
= \<[\xi,\eta]_\gm,\xi\> + \<\eta, [\xi,\xi]_\gm\>
= \<[\xi,\eta]_\gm,\xi\>$, so naturally reductive spaces are geodesic orbit 
spaces.  Or one can see this by noting that (\ref{def-nat}) is the case
$\alpha = 0$ of (\ref{basic}).

As in (\ref{def-poly}) one has corresponding polynomials 
\begin{equation}\label{def-nat-poly}
\begin{aligned}
&\psi: \gm +\gm + \gm \to \R \text{ by }
        \psi(\xi,\eta,\zeta) = 
	\<[\xi,\eta]_\gm,\zeta\> + \<\eta, [\xi,\zeta]_\gm\> \text{ and } \\
&\psi_\C: \gm_\C + \gm_\C + \gm_\C \to \C \text{ by }
        \psi_\C(\xi,\eta,\zeta) = 
	\<[\xi,\eta]_{\gm_\C},\zeta\> + \<\eta, [\xi,\zeta]_{\gm_\C}\>.
\end{aligned}
\end{equation}
As in (\ref{def-var}) those polynomials define corresponding moduli spaces
\begin{equation}\label{def-nat-var}
\begin{aligned}
&\Psi = \{(\xi,\eta, \zeta)) \in (\gm + \gm + \gm) \mid
        \psi(\xi,\eta,\zeta) = 0\}
        \text{ and }\\
&\Psi_\C = \{(\xi,\eta, \zeta) \in (\gm_\C + \gm_\C + \gm_\C) \mid
        \psi_\C(\xi,\eta,\zeta) = 0\}.
\end{aligned}
\end{equation}
Then we have the analog of Proposition \ref{realform}:

\begin{proposition}\label{nat-realform}
The real affine variety $\Psi = \Psi_\C \cap (\gm + \gm + \gm)$, and
it is a real form of the complex affine variety $\Psi_\C$.  In other
words $\Psi_\C$ is the complexification of $\Psi$.
In particular, if $f$ is a holomorphic function on $\Psi_\C$ and
$f|_\Psi \equiv 0$, then $f \equiv 0$.
\end{proposition}
We reformulate the definition (\ref{nat-basic}) of naturally reductive space:

\begin{proposition}\label{natred-space}
In the notation of {\rm (\ref{def-nat-var})},
$(G/H,ds^2)$ is a naturally reductive space if and only if
$\psi(\xi,\eta,\zeta) \in \Psi$ whenever  
$(\xi,\eta, \zeta)) \in (\gm + \gm + \gm)$.
The crown $\crown\{\{(G/H,ds^2)\}\}$ is a naturally reductive space if
and only if $\psi(\xi,\eta,\zeta) \in \Psi_\C$ whenever 
$(\xi,\eta, \zeta)) \in (\gm_\C + \gm_\C + \gm_\C)$.
\end{proposition}

Now we combine Propositions \ref{nat-realform} and \ref{natred-space}:

\begin{theorem}\label{natred-theorem}
In the notation of {\rm (\ref{def-nat-var})},
$(G/H,ds^2)$ is a naturally reductive space if and only if
the crown $\crown\{\{(G/H,ds^2)\}\}$
is a naturally reductive space.  In particular, if $(G'/H',ds'^2) \in
\{\{(G/H,ds^2)\}\}$, then $(G/H,ds^2)$ is a naturally reductive space if
and only if $(G'/H',ds'^2)$ is a naturally reductive space.
\end{theorem}

\section{Commutative Spaces}\label{sec5}
\setcounter{equation}{0}

Consider a pseudo-riemannian manifold $(G/H,ds^2)$ where $G$ is the 
identity component of the group of all isometries.  The $G$--invariant
differential operators on $G/H$ form an associative algebra $\cD(G,H)$.
We say that $(G/H,ds^2)$ is {\bf commutative} if the algebra $\cD(G,H)$ 
is commutative.  This is the usual definition when $H$ is compact and 
$(G/H,ds^2)$ is riemannian, but it makes perfectly
good sense (and is appropriate for us) in any signature.

We will discuss D'Atri spaces in Section \ref{sec7}, but the point here is that
commutative spaces are D'Atri spaces \cite{KV1983}.  In dimensions $\leqq 5$ 
one can say a bit more.  There,
a homogeneous riemannian manifold is commutative if and only if it is
naturally reductive. 

Without loss of generality we suppose that $G/H$ is connected and simply 
connected, and that $G = I^0(G/H,ds^2)$.  Then also $H$ is connected.
As before we start with a reductive decomposition $\gg = \gh + \gm$.
Identify $\gm$ with the tangent space $T_{x_0}(G/H)$ at the base point
$x_0 = 1H \in G/H$.  That gives 
an obvious $\Ad(H)$--equivariant bijection between $\cD(G,H)$
and the $\Ad(H)$--invariants \footnote{by an abuse of notation
we write $\Ad(H)$ instead of $S(\Ad(H))$ for 
the symmetric powers that form the action of $H$ on $S(\gm)$.}
$S(\gm)^H$ in the symmetric algebra $S(\gm)$.
See Helgason, \cite[Ch. II, Theorem 4.6]{H1984}, for the details.

Given any real basis $\{\zeta_1, \dots , \zeta_\ell\}$ of $\gh$,  
$S(\gm)^H$ is the intersection 
$\bigcap_{1 \leqq j \leqq \ell} S(\gm)^{\zeta_j}$ of null spaces of the 
$\ad(\zeta_j)|_{S(\gm)}$.  As $\{\zeta_j\}$ is a complex basis of
$\gm_\C$ we have

\begin{lemma}\label{comm-basis}
The algebra $\cD(G_\C,H_\C)$ is the complexification $\cD(G,H)_\C$ of
the algebra.  $\cD(G,H)$.  
\end{lemma}

Now $\cD(G_\C,H_\C)$ is commutative if and only if $\cD(G,H)$ is commutative.
We apply this to real form families.

\begin{theorem}\label{comm-family}
The pseudo--riemannian manifold $(G/H,ds^2)$ is commutative if and only
if its complexification $(G_\C/H_\C, ds_\C^2)$ is commutative.  If
$(G'/H',ds'^2) \in \{\{(G/H,ds^2)\}\}$, then $(G'/H',ds'^2)$ is
commutative if and only if $(G/H,ds^2)$ is commutative.
\end{theorem}

\section{Weakly Symmetric Spaces}\label{sec6}
\setcounter{equation}{0}

Recall that a pseudo--riemannian manifold $(M,ds^2)$ is {\bf weakly symmetric}
if, given $x \in M$ and a tangent vector $\xi \in T_x(M)$, there is an
isometry $s_{x,\xi} \in I(M,ds^2)$ such that $s_{x,\xi}(x) = x$ and
$ds_{x,\xi}(\xi) = -\xi$.  The familiar special case: $(M,ds^2)$ is symmetric 
if, given $x \in M$
there is an isometry $s_x \in I(M,ds^2)$ such that $s_x(x) = x$ and
$ds_x(\xi) = -\xi$ for every $\xi \in T_x(M)$.  

Riemannian weakly symmetric spaces were introduced by Selberg \cite{S1956}
in the context of harmonic analysis and algebraic geometry.  One of
his results was that riemannian weakly symmetric spaces are commutative.
In view of Theorem \ref{comm-family},

\begin{corollary}\label{gen-selberg}
Let $(G/H,ds^2)$ be a riemannian weakly symmetric space.  Then
$\crown\{\{(G/H,ds^2)\}\}$ is commutative, and every
$(G'/H',ds'^2) \in \{\{(G/H,ds^2)\}\}$ is commutative.
\end{corollary}

Weakly symmetric pseudo--riemannian manifolds are geodesic orbit spaces
\cite[Theorem 4.2]{CW2012}.  Thus, if $(G/H,ds^2)$is weakly symmetric then, by 
Theorem \ref{GO-theorem}, every $(G'/H',ds'^2) \in \{\{(G/H,ds^2)\}\}$ is
a geodesic orbit space.  

There are $\aleph_0$ examples in the tables of \cite{CW2017} and
\cite{WC2018}.  Tables 3.6, 4.12, 5.1, 5.2 and 5.3 in \cite{CW2017}
list various classes of real form families $\{\{(G/H,ds^2)\}\}$
with $(G/H,ds^2)$ weakly symmetric, $G$ semisimple and $H$ reductive.
The Tables in \cite{WC2018} list various classes of real
form families $\{\{(G/H,ds^2)\}\}$ for which $(G/H,ds^2)$ is a weakly
symmetric nilmanifold with $G = N \rtimes H$.

The question here is just when weak symmetry of $(G/H,ds^2)$ implies weak
symmetry for the members of its real form family $\{\{(G/H,ds^2)\}\}$. 
A partial answer is implicit in a result of Akhiezer and Vinberg 
\cite[Theorem 12.6.10]{W2007}; see \cite[Corollary 12.6.12]{W2007}:

\begin{proposition}\label{wsym-red}
Let $(G/H,ds^2)$ be a weakly symmetric pseudo--riemannian manifold with
$G$ connected and reductive, and $H$ reductive in $G$.  Then every
$(G'/H',ds'^2) \in \{\{(G/H,ds^2)\}\}$ is weakly symmetric.
\end{proposition}

See \cite[Section 15.4]{W2007} for a discussion of commutativity for
weakly symmetric riemannian nilmanifolds.

\section{D'Atri Spaces}\label{sec7}
\setcounter{equation}{0}

We say that
a pseudo--riemannian manifold $(M,ds^2)$ is a {\bf D'Atri space}
if its local geodesic symmetries $\sigma_x: \exp(t\xi) \mapsto \exp(-t\xi)$,
$\xi \in T_x(M)$ and $t$ reasonably small, are volume preserving.  This is
the standard definition when $(M,ds^2)$ is riemannian, but it makes perfectly
good sense (and is appropriate for us) in any signature.

A geodesic orbit riemannian manifold is a D'Atri space 
\cite[Theorem 1]{KV1984}.  That argument of Kowalski and Vanhecke goes
through {\em mutatis mutandis} for pseudo--riemannian manifolds, using 
the definition introduced just above.   Or see \cite{KV1985} to
develop this in the more general setting of two-point functions. 
In any case, we now have inclusions of real form families of 
pseudo--riemannian manifolds:
\begin{equation}\label{inclusions}
\begin{aligned}
&\text{(weakly symmetric spaces) } \subset
	 \text{ (geodesic orbit spaces) } \subset 
	 \text{ (D'Atri spaces)} \\
&\text{(naturally reductive spaces) } \subset
	 \text{ (geodesic orbit spaces) } \subset 
         \text{ (D'Atri spaces)}
\end{aligned}
\end{equation}
This suggests a number of open problems, one of which was noted toward 
the end of Section \ref{sec6}:
\begin{itemize}
\item If $(G/H, ds^2)$ is weakly symmetric and $(G'/H',ds'^2) \in
	\{\{(G/H, ds^2)\}\}$, is $(G'/H/.ds'^2)$ weakly symmetric?
\item Can the naturally reductive weakly symmetric spaces be
characterized as the weakly symmetric spaces $(G/H,ds^2)$ for which 
every $(G'/H',ds'^2) \in \{\{(G/H, ds^2)\}\}$ is weakly symmetric?
\item If $(G/H, ds^2)$ is a D'Atri space and $(G'/H',ds'^2) \in
        \{\{(G/H, ds^2)\}\}$, is $(G'/H',ds'^2)$ a D'Atri space?
\item Can the geodesic orbit spaces be
characterized as the D'Atri spaces $(G/H,ds^2)$ for which 
every $(G'/H',ds'^2) \in \{\{(G/H, ds^2)\}\}$ is a D'Atri space?
\item Which commutative spaces are weakly symmetric spaces?
\item What happens if we restrict these questions to the case of
spaces $(G/H,ds^2)$ for which $G$ is semisimple (or real reductive) 
and $H$ is reductive in $G$?
\item What happens if we restrict these questions to the case of
spaces $(G/H,ds^2)$ for which $G$ is of the form $N \rtimes H$
with $N$ nilpotent?
\end{itemize}

\end{document}